\documentclass[12pt,a4]{article}
\newtheorem{theorem}{Theorem}
\newtheorem{lemma}{Lemma}
\newtheorem{corollary}{Corollary}
\newtheorem{definition}{Definition}
\newtheorem{example}{Example}
\newtheorem{example*}{Example*}
\newtheorem{proposition}{Proposition}
\newtheorem{remark}{Remark}
\newenvironment{proof}[1][Proof]{\textbf{#1.} }{\ \rule{0.5em}{0.5em}}

\usepackage[square]{natbib}
\usepackage{amssymb} 
\usepackage{phonetic} 
\usepackage{amsmath}
\usepackage{algorithm}
\usepackage{algpseudocode}
\begin{document}
\title{Formal Solutions of a Class of Pfaffian Systems in Two Variables\footnote{Submitted to the \textit{International Symposium on Symbolic and Algebraic Computation}, Kobe, Japan, 2014.}}

\author {\textbf{Moulay Barkatou, Suzy S. Maddah} \footnote{Enrolled  under a joint PhD program with the Lebanese University}\\
     XLIM UMR 7252 ; DMI\\
          University of Limoges; CNRS\\
       123, Avenue Albert Thomas\\
       87060 Limoges, France\\
       moulay.barkatou@unilim.fr\\
       suzy.maddah@etu.unilim.fr\\[10pt]
         \textbf{Hassan Abbas}\\
       Laboratory of Mathematics\\
       Lebanese University\\
       Beirut, Lebanon\\
       habbas@ul.edu.lb\\
}

\maketitle
\begin{abstract}
In this article, we present an algorithm which computes a fundamental matrix of formal solutions of completely integrable Pfaffian systems with normal crossings in two variables, based on \citep{key24}. A first step was set in \citep{key5} where the problem of rank reduction was tackled via the approach in \citep{key57}. We give instead a Moser-based approach following \citep{key102,key41}. And, as a complementary step, we associate to our problem a system of ordinary linear singular differential equations from which the formal invariants can be efficiently derived via the package ISOLDE \citep{key27}, implemented in the computer algebra system Maple. 
\end{abstract}

\newpage
\section{Introduction}
\label{introduction}
Let $\mathcal{O}= \mathbb{C}[[x_1, x_2, \dots x_m]]$ be the ring of formal power series in $x=(x_1, x_2,\dots, x_m)$ over the field of complex numbers $\mathbb{C}$ and $K$ its field of fractions. Let $V$ be a $K$-vector space of dimension $n$, $A^{(1)} , \dots , A^{(m)}$ be nonzero matrices with entries in $\mathcal{O}$, and $p_1, \dots, p_m$ be nonnegative integers. Set $\delta_i =x_i \frac{\partial}{\partial x_i}$. Then the operator $\Delta_i = \delta_i  - A^{(i)} / x^{p_i}_i, \; i=1, \dots, m, \;$ is a $\delta_i$-differential operator acting on $V$, that is, an additive map from $V$ to itself satisfying the Leibniz condition: 
$$\forall f \in \mathcal{O} ,\; v \in V ,\; \Delta_i(fv) = \delta_i (f) v + f \Delta_i(v) .$$
Let $Y$ be an unknown $n$-dimensional column vector. In this article, we are interested in the formal reduction of the so-called \textit{completely integrable Pfaffian systems with normal crossings}, the class of linear systems of partial differential equations in $m$ variables and dimension $n$, given by
\begin{equation} \label{originalpfaffoperator} \Delta_i Y = 0 ,\; 1 \leq i \leq m . \end{equation}
satisfying the integrability conditions (pairwise commutativity of the operators), 
$$\label{conditionpfaffoperator} \Delta_i \circ \Delta_j =   \Delta_j \circ \Delta_i ,\quad  1 \leq i, j \leq m .$$
Pfaffian systems arise in many application \citep{key74} including the studies of aerospace and celestial mechanics \citep{key31}. By far, the most important for applications are those with normal crossings \citep{key32}. The associated pfaff (1-form) to system \eqref{originalpfaffoperator} is $d= \sum_{i=1}^{m} \frac{A^{(i)}}{x_i^{p_i+1}} dx_i$. Without loss of generality, the singularity of system\eqref{originalpfaffoperator} is placed at the origin. Otherwise, translations in the independent variables can be performed. The singular locus of the system is the union of hyperplanes of coordinates \; $x_1 x_2 \dots x_m =0$. This is what is referred to as \textit{normal crossings}. 

Let $T \in GL(V)$. A change of basis $Y=T Z$ gives rise to 
\begin{equation} \label{equivalentpfaff} \tilde{\Delta}_i Z = 0, \; 1 \leq i \leq m \end{equation}
where $ \tilde{\Delta}_i = T^{-1} {\Delta}_i T = \delta_i - \tilde{A}^{(i)} / x^{\tilde{p}_i}_i$ and
\begin{equation}  \label{gaugepfaff} \frac{\tilde{A}^{(i)} }{ x^{\tilde{p}_i}_i} = T^{-1}  ( \frac{A^{(i)}}{ x^{p_i}_i} T - \delta_i T), \quad 1 \leq i \leq m .\end{equation}
We say that system \eqref{equivalentpfaff} is \textit{equivalent} to system \eqref{originalpfaffoperator}. Without loss of generality, we assume that $A^{(i)}(x_i =0)$ are nonzero matrices; otherwise $p_i$ can be readjusted. The m-tuple $(p_1 , \dots , p_{m})$ of nonnegative integers is called the \textit{Poincar\'{e} rank} of the system. It is said to be minimal, and is called the \textit{true Poincar\'{e} rank}, whenever the $p_i$'s are simultaneously and individually the smallest among all the possible changes of basis $Y= T Z ,\; T \in GL_n(K)$. 

In \citep{key53,key4}, the language of stable modules over the ring of power series is used to establish the following theorem (Theorem 1 in \citep{key53} and Main Theorem in \citep{key4}).

\begin{theorem}
\label{gerardexistence}
Consider the completely integrable system \eqref{originalpfaffoperator}. Under additional ramifications, $x_i = t_i^{\alpha_i}$ where $\alpha_i$ is a positive integer, $1 \leq i \leq m$, there exists $T$ that belongs to $GL_n(\mathbb{C}((t_1, \dots, t_m)))$ such that the change of basis $Y= T Z$ gives the \textit{equivalent} system:
 $$ T^{-1} \Delta_i T = \frac{\partial}{\partial t_i} - \tilde{A}^{(i)}(t_i)$$ 
where $\tilde{A}^{(i)}(t_i)  =  Diag ( \tilde{A}^i_1,  \tilde{A}^i_2, \dots, \tilde{A}^i_h ) ,$ and for all $1 \leq k \;\leq h$ we have
\begin{itemize}
\item $ \tilde{A}^i_k = w_k^i (t^{-1}_i) I_{n_k} + t_i^{-1} N^i_k $;
\item $w_k^i (t_i) = \sum_{j=1}^{\alpha_i p_i +1} \xi_{ijk} t_i^{-j}$ is a polynomial in $t^{-1}_i$ with coefficients in $\mathbb{C}$. In particular, $\bar{w}_k^i (t_i) = \sum_{\textbf{j=2}}^{\alpha_i p_i +1} \xi_{ijk} t_i^{-j}$, for all $1 \leq k \;\leq h$, are called the $x_i$-\textit{exponential parts}; 
\item $N^i_k$ is a square matrix of order $n_k$ with elements in $\mathbb{C}$,  nilpotent in upper triangular form;
\item $n_1 + n_2 + \dots + n_h = n$.
\end{itemize}
\end{theorem}
This theorem guarantees the existence of a change of basis which takes system \eqref{originalpfaffoperator} to Hukuhara-Turritin's normal form from which the construction of a fundamental matrix of formal solutions (see \eqref{solutionpfaff2} for $m=2$) is straightforward.  In particular, if two systems are \textit{equivalent} then they have the same $x_i$-\textit{exponential parts} for all $1 \leq i \leq m$. However, the formal reduction, that is to say the algorithmic procedure computing such a change of basis, is a question of another nature.

We remark that the integrability conditions and the property of normal crossings play a major role in establishing this theorem. In particular, they give rise to the following aspect of system \eqref{originalpfaffoperator} (Proposition 1, page 8, \citep{key4}).
\begin{proposition}
\label{eigenvaluespfaff}
The eigenvalues of $A^{(i)}(x_i=0), \; 1 \leq i \leq m$, lie in the field of constants.
\end{proposition}
The univariate case, $m=1$, is the well-known case of linear singular system of ordinary differential equations (ODS, in short) which have been studied extensively (see, e.g., \citep{key6,key7} and references therein). Moreover, unlike the case of $m >1$, algorithms to related problems leading to the construction of the formal solutions and computation of the \textit{exponential parts} have been developed by various authors (see, e.g., \citep{key24,key40,key25,key26} and references therein). The package ISOLDE  \citep{key27} written in the computer algebra system Maple is dedicated to the symbolic resolution of ODS and more generally linear functional matrix equations.

The bivariate case, $m=2$, is the interest of this article. Our first contribution is an explicit method to compute the $x_i$-\textit{exponential parts}. Alongside their importance in the asymptotic theory, their computation reduces the ultimate task of formal reduction to constructing a basis of the $\mathbb{C}$-space of regular solutions (see, for $m=1$, \citep{key25}). Rank reduction as well can complement our work into full formal reduction. Rank reduction is the explicit computation of a change of basis which takes system \eqref{originalpfaffoperator} to an \textit{equivalent} system whose rank is the true Poincar\'{e} rank. Two well-known approaches in the uniariate case are those of Levelt \citep{key57} and Moser \citep{key19}. The former is generalized for systems \eqref{originalpfaffoperator} in two variables ($m=2$) in \citep{key5}. The latter, on the other hand, gives a \textit{reduction criterion} upon which the efficient algorithms of \citep{key27} are based. It is generalized for singularly-perturbed linear differential systems in \citep{key102} and examples over the bivariate field  favored its efficiency [Section 4 in \citep{key102}]. Our second contribution is a Moser-based rank reduction algorithm of system \eqref{originalpfaffoperator} in two variables. This establishes its formal reduction upon following the univariate-case descripiton of \citep{key24}. 

We remark that the multivariate particular case of $p_i =0$ for all $1 \leq i \leq m$ (systems with singularity of first kind) is studied in \citep{key73} and references therein. 

Our article is organized as follows: 
In Section \ref{preliminariespfaff} we give the preliminaries and restrict our notations to two variables. In Section \ref{reviewods} we review the formal reduction techniques in \citep{key24} of ODS, including the computation of the \textit{exponential parts}. Thereafter the discussion is restricted to the bivariate case. In Section \ref{formalreduction} we present the building 
blocks of formal reduction reducing the discussion to two elements. The first of which is the computation of the $x_i$-\textit{exponential parts} given in Section \ref{expopolymatrices}. The second is the Moser-based rank reduction algorithm given in Section \ref{moser}. We finally discuss the limitations and point out prospects of further investigations in Section \ref{conclusion}. Our main results are Theorem \ref{exponentialpfaff} and Theorem \ref{moserpfaffian}.

\section{Preliminaries and Notations}
\label{preliminariespfaff}
For the clarity of presentation in the bivariate case, we set $\mathcal{O}=\mathbb{C}[[x,y]]$, $K$ its field of fractions, $\mathcal{O}_x=\mathbb{C}[[x]],$ and $\mathcal{O}_y=\mathbb{C}[[y]]$. The variables $x$ and $y$ will be dropped from the notation whenever ambiguity is not likely to arise. We refer to the identity and zero matrices of prescribed dimensions by $I_\nu$ and $O_{\nu,\mu}$. System \eqref{originalpfaffoperator} can be rewritten as
\begin{equation}
\label{originalpfaff2}
\begin{cases}
x \frac{\partial Y}{\partial x} = A (x,y) Y= x^{-p} (A_0(y) + A_1(y) x+ A_2(y) x^2 + \dots) Y\\ 
y \frac{\partial Y}{\partial y} = B (x,y)  Y=  y^{-q} ({B}_0(x) + {B}_1(x) y + {B}_2(x) y^2 + \dots) Y
\end{cases} 
\end{equation}
where ($A_0(y)$, $B_0(x)$) is called the leading coefficient pair. The Poincar\'{e} rank of system \eqref{originalpfaff2} is thus $(p,q)$ and we refer to $(A_{00}, B_{00}):=(A(0,0), B(0,0))$ as the leading constant pair.
 An easy calculation shows that the the integrability condition is given by 
\begin{equation}
\label{conditionpfaff2}
x \frac{\partial  B}{\partial x} +  B A = y \frac{\partial  A}{\partial y} +  A B .  \end{equation}
It follows from Theorem \ref{gerardexistence} that a fundamental matrix of formal solutions has the following form and properties
\begin{equation} \label{solutionpfaff2} \Phi(x^{1/s_1}, y^{1/s_2} ) \; x^{{\Lambda}_1} \; y^{{\Lambda}_2} \; exp(Q_1(x^{-1/s_1})) \; exp(Q_2(x^{-1/s_2})) \end{equation}
\begin{itemize}
\item $s=(s_1,s_2)$ is two-tuple of positive integers;
\item $\Phi$ is an invertible meromorphic series in $(x^{1/s_1}, y^{1/s_2})$ over $\mathbb{C}$;
\item $Q_1, Q_2$ are diagonal matrices containing polynomials in $x^{-1/s_1}$ and $y^{-1/s_2}$ over $\mathbb{C}$ without contant terms; they are obtained by formally integrating the $x,y$-\textit{exponential parts} respectively;
\item $\Lambda_1$ and $\Lambda_2$ are constant matrices commuting with $Q_1$ and $Q_2$.
\end{itemize}
System \eqref{originalpfaff2} (resp. system \eqref{originalpfaffoperator}) is said to be \textit{regular singular} whenever $Q_1$ and $Q_2$ are null. In this case $s=(1,1)$, and the formal series $\phi$ converges whenever the series of $A$ and $B$ do. Otherwise, system \eqref{originalpfaff2} is \textit{irregular singular} and the entries of $Q_1, Q_2$ determine the main asymptotic behavior of the actual solutions as $x_i \rightarrow 0$ in appropriately small sectorial regions (Proposition 5.2, page 232, and Section 4 of \citep{key1}). It is shown in \citep{key20,key21},  for the multivariate case in geometric and algebraic settings respectively, that the regularity of system \eqref{originalpfaff2} is equivalent to the regularity of the individual subsystems each considered as a system of ordinary differential equations. As a consequence, system \eqref{originalpfaff2} is \textit{regular singular} if and only if its true Poincar\'{e} rank is $(0,0)$. And to test its regularity, algorithms given for $m=1$ (e.g.\citep{key26,key57}) can be applied separately to the individual subsystems. 

We end this section by a further characterization of a change of basis $T$ in the bivariate case, which takes system \eqref{originalpfaff2} to an \textit{equivalent} system in a weak-triangular form, rather than that of Theorem \ref{gerardexistence}. However, we'll see in Theorem \ref{exponentialpfaff} that this form suffices to give an insight into the computation of $x,y$-\textit{exponential parts}. The following Proposition is given as [Proposition3, page 654, \citep{key53}] and the proof is to be omitted here due to the lack of space. 
\begin{proposition}
\label{gerardtransformation}
Consider the completely integrable system \eqref{originalpfaff2}. Under additional ramification, $x = t^{s_1}$ where $s_1$ is a positive integer, there exists $T \in GL_n(\mathbb{C}((t,y)))$, product of transformations of type $diag(t^{k_1}, \dots, t^{k_n})$ (where $k_1, \dots, k_n$ are nonnegative integers) and transformations in $GL_n(\mathbb{C}[[t,y]])$, such that the change of basis $T= Y Z$ gives the following \textit{equivalent} system: 
\begin{equation} \label{transformedpfaff2} \begin{cases}
\frac{\partial Z}{\partial t} = \tilde{A} (t,y) Z  \\
 \frac{\partial Z}{\partial y} = \tilde{B} (y) Z \end{cases} \end{equation}
where $\tilde{A} (t,y)= Diag ( \tilde{A}_1,\tilde{A}_2, \dots, \tilde{A}_h)$ and for all $1 \leq k \leq h $
\begin{itemize}
\item $\tilde{A}_k = a_k (t^{-1}) I_{n_k} + t^{-1} N_k (y)$;
\item $a_k (t)=\sum_{j=1}^{s_1 p +1} \xi_{jk} t^{-j}$ is a polynomial in $t^{-1}$ with coefficients in $\mathbb{C}$; 
\item $N_k$ is a  $n_k$-square matrix with elements in $\mathcal{O}_y$;
\item $y^{q+1} \tilde{B}(y) \in \mathcal{O}_y$ and  $\tilde{B}(y)= Diag ( \tilde{B}_1,\tilde{B}_2, \dots, \tilde{B}_h)$;
\item $y^{q+1} \tilde{B}_k |_{y=0}= c_k I_{n_k} + F_k$ where $c_k \in \mathbb{C}$ and $F_k$ is a nilpotent constant matrix.
\end{itemize}
\end{proposition}
\begin{remark}
\label{exp}
Let $\bar{a}_k (t) = \sum_{\textbf{j=2}}^{s_1 p+1} \xi_{jk} t^{-j}$ then they are the $x$-\textit{exponential parts} of system \eqref{transformedpfaff2} (resp. system \eqref{originalpfaff2}). 
\end{remark}
\section{Formal Reduction of ODS}
\label{reviewods}
For $m=1$, system \eqref{originalpfaffoperator} reduces to the linear singular differential system of ordinary equations
\begin{equation}
\label{originalods}
x \frac{dY}{d x} = A (x) Y = x^{-p}  (A_0 + A_1 x + A_2 x^2 + \dots) Y
\end{equation}
where $A_0 := A(0)$ is the leading coefficient matrix and $p$ is a nonnegative integer denoting the \textit{Poincar\'e rank}. For a fundamental matrix of solutions, in analogy to \eqref{solutionpfaff2}, we write
\begin{equation} \label{solutionods} \Phi(x^{1/s}) \; x^{{\Lambda}} \; exp(Q (x^{-1/s})). \end{equation}
If $T \in GL_n(\mathbb{C}((x)))$ then the change of basis $Y = T Z$ results in the equivalent system 
\begin{equation}
\label{equivalentods}
x \frac{dZ}{d x} = \tilde{A} (x) Z= x^{-\tilde{p}} (\tilde{A}_0 + \tilde{A}_1 x + \tilde{A}_2 x^2 + \dots) Z. 
\end{equation}
The second author of this article developed in \citep{key24} a recursive algorithmic process that constitutes of finding, at every step, a change of basis which results in a system \eqref{equivalentods} \textit{equivalent} to system \eqref{originalods} which is either of lower \textit{Poincar\'e rank} or can be decoupled into systems of lower dimensions. At each step, the starting point is the nature of eigenvalues of the leading coefficient matrix $A_0$ according to which either a block-diagonalized equivalent system, a Moser-irreducible one, or the Katz invariant are computed. The following operations summarize the recursive process:
\begin{itemize}
\item \textbf{Block-diagonalization}: Whenever $A_0$ has at least two distinct eigenvalues, system \eqref{originalods} can be decoupled into systems of lower dimensions via the classical Splitting Lemma. For its statement and a constructive proof one may consult [Section 12, pages 52-54, \citep{key7}]. 
\item \textbf{Eigenvalue shifting}: Whenever $A_0$ has a single nonzero eigenvalue $\gamma \in \mathbb{C}$, the so-called Eigenvalue shifting
\begin{equation} \label{eigenvalueshifting} Y = exp ( \int \gamma x^{-p-1} dx) Z, \end{equation}
results in a system with a nilpotent leading coefficient matrix (in fact, $\tilde{A}(x)= A - x^{-p} \gamma I_n).$
\item \textbf{Moser-based rank reduction}: This is the rank reduction based on the reduction criterion defined by Moser in \citep{key19}. It results in an equivalent system whose \textit{Poincar\'e rank} is the  \textit{true Poincar\'e rank} and whose leading coefficient matrix has the minimal rank among any possible choice of a change of basis. For an efficient Moser-based rank reduction algorithm, one may consult \citep{key26}. If the leading matrix coefficient of the \textit{equivalent} Moser-irreducible system is still nilpotent, one proceeds to compute the Katz invariant, a process for which Moser-irreducibility is a prerequisite. 
\item \textbf{Katz invariant}: 
\begin{definition}
\label{katz}
Given system \eqref{originalods} whose $x$-\textit{exponential parts} are denoted by  ${\{\bar{w}_k\}}_{1 \leq k \leq n}$. For $1 \leq k \leq n$, let $\zeta_k$ be the minimum exponent in $x$ within the terms of $\bar{w}_k$.   The \textit{Katz invariant} of \eqref{originalods} (resp. $A(x)$) is then the rational number 
$$\kappa = - 1 - {min}_{1 \leq k \leq n} \; \zeta_k.$$
\end{definition}
\textit{Katz invariant} can be obtained from the characteristic polynomial of $A(x)$, i.e. $det(\lambda I - A(x))$, given that $A(x)$ is Moser-irreducible [e.g. Theorem 1 in \citep{key24}]. Consider a Moser-irreducible system \eqref{originalods} whose leading matrix coefficient is nilpotent and $\kappa= \frac{l}{m}$ with $l,m$ relatively prime positive integers. Then, a ramification, that is to say a re-adjustement of the independent variable $t = x^{1/m}$, followed by Moser-based rank reduction results in an \textit{equivalent} Moser-irreducible system whose \textit{Poincar\'e rank} is equal to $l$ and its leading matrix coefficient has at least $m$ distinct eigenvalues. Hence, block-diagonalization may be applied again. 
\end{itemize}
This recursive process results either in a group of decoupled systems with dimension $n=1$ (scalar case) or \textit{Poicar\'e rank} $p=0$ (system with singularity of first kind).  For the latter case one may consult Chapter 1 in \citep{key7} or \citep{key25} for a more general context. The changes of basis applied at every stage are used to construct a fundamental matrix of solutions \eqref{solutionods}. Algorithms of the four described operations are implemented in Maple (see ISOLDE \citep{key27}).
\section{Formal Reduction in the bivariate case}
\label{formalreduction}
Consider again the completely integrable system \eqref{originalpfaff2}
$$
\begin{cases}
x \frac{\partial Y}{\partial x} = A (x,y) Y= x^{-p} (A_0(y) + A_1(y) x + \dots) Y\\ 
y \frac{\partial Y}{\partial y} = B (x,y)  Y=  y^{-q} ({B}_0(x) + {B}_1(x) y + \dots) Y.
\end{cases} $$
A major difficulty within the symbolic manipulation of system \eqref{originalpfaff2} (resp. system \eqref{originalpfaffoperator}) arises from \eqref{gaugepfaff} as it is evident that any transformation applied to any of the subsystems alters the others. Hence, the generalization of the univariate-case techniques is not straightforward. In particular, the equivalent system does not necessarily inherit the normal crossings even for very simple examples, as exhibited by Example \ref{exmnaive} in [Section 4, \citep{key5}] which we recall here.
\begin{example}
\label{exmnaive}
Consider the following completely integrable pfaffian system with normal crossings of \textit{Poinca\'e rank} $(3,1)$ and \textit{true Poincar\'e rank} $(0,0)$.
\begin{equation} \label{exmnaivesys} \begin{cases}
x \frac{\partial Y}{\partial x} = A(x, y) Y = x^{-3}\begin{bmatrix} x^3 +y & y^2 \\ -1 & -y + x^3 \end{bmatrix}  Y\\
y \frac{\partial Y}{\partial y} = B(x, y) Y =  y^{-1} \begin{bmatrix} y & y^2 \\ -2 & -3\end{bmatrix} Y.
\end{cases} . \end{equation}
The change of basis $Y= \begin{bmatrix} x^3 & -y^2 \\ 0 & y \end{bmatrix} Z$ computed by the univariate-case Moser-based rank reduction algorithm, upon regarding the first subsystem as an ODS in $x$,  results in the following equivalent system
$$\begin{cases}
\label{gaugepfaffian}
x \frac{\partial Z}{\partial x} = \tilde{A} (x,y) Z= \begin{bmatrix} -2& 0 \\ \frac{-1}{y} & 1 \end{bmatrix}  Z\\
y \frac{\partial Z}{\partial y} = \tilde{B}(x, y) Z =  y^{-2} \begin{bmatrix} -y^2 & 0 \\ -2 x^3 & -2 y^2 \end{bmatrix} Z.
\end{cases} $$\\
We can see that such a transformation achieves the goal of diminishing the rank of the first subsystem, considered as an ODS, to its minimum ($\tilde{p}=0$). However, it alters the normal crossings as it introduces the factor $y$ in the denominator of an entry in $\tilde{A}$. Moreover, it elevates the rank of the second subystem.
\end{example}
The urge to preserve the normal crossings, whose importance is highlighted in the Introduction, motivates the following definition
 \begin{definition}
\label{compatible}
Let $T \in GL_n(K)$. We say that the change of basis $Y = T Z$ (resp. $T$) is \textit{compatible} with system \eqref{originalpfaff2} if the normal crossings of the system is preserved and the Poincar\'{e} rank of the individual subsystems is not elevated. 
\end{definition}
\begin{remark}
Clearly, if $T$ is a constant matrix or it lies in $GL_n(\mathcal{O})$ then it is \textit{compatible} with system \eqref{originalpfaff2}. 
\end{remark}
As in the univariate case, the main difficulties in formal reduction arise whenever the leading constant pair consists of nilpotent matrices. Since otherwise a block diagonalization can be attained. A generalization of the univariate-case Splitting Lemma is given with constructive proof in [Section 5.2, page 233, \citep{key1}]. We repeat hereby the theorem without proof.
\begin{theorem}
\label{blockpfaff}
Given system \eqref{originalpfaff2} with leading constant pair
$$A_{00} = diag(A_{00}^1, A_{00}^2) \; \text{and}\; B_{00} = diag(B_{00}^1, B_{00}^2).$$ 
If the matrices in one of the couples $(A_{00}^1, A_{00}^2)$ or $ (B_{00}^1, B_{00}^2)$ have no eigenvalues in common, then there exists a unique transformation $T \in GL_n(\mathcal{O})$ partitioned conformally
$$T(x,y) = \begin{bmatrix} I & T^{12} \\ T^{21} & I \end{bmatrix}$$
such that the change of basis $Y = T Z$ results in the following \textit{equivalent} system partitioned conformally with $A_{00}, B_{00},$ 
$$
\begin{cases}
x \frac{\partial Z}{\partial x} = \tilde{A} (x,y) Z =  diag(\tilde{A}^1, \tilde{A}^2) Z\\
y \frac{\partial Z}{\partial y} = \tilde{B} (x,y) Z =  diag(\tilde{B}^1, \tilde{B}^2) Z
\end{cases} $$
where $(\tilde{A}^1_{00},\tilde{A}^2_{00})= (A_{00}^1, A_{00}^2)$ and $(\tilde{B}^1_{00},\tilde{B}^2_{00})= (B_{00}^1, B_{00}^2)$.
\end{theorem}
By an eigenvalue shifting \eqref{eigenvalueshifting} we can then arrive at an equivalent system whose leading constant pair consists of nilpotent matrices. It can be easily verified that such a transformation is \textit{compatible} with system \eqref{originalpfaff2}. By Proposition \ref{eigenvaluespfaff}, the matrices of the leading coefficient pair are nilpotent as well.
The case of $n=1$ is straightforward and the case of $(p,q)=(0,0)$ is already resolved in [Chapter 3, \citep{key73}]. Henceforth, following Section \ref{reviewods}, the problem of formal reduction is reduced now to discussing two operations: computing Katz invariant and a \textit{compatible} Moser-based rank reduction. The former is the subject of the next section.

\section{Computing Exponential Parts and Katz Invariant}
\label{expopolymatrices}
In analogy to Definition \ref{katz}, the Katz invariant of system \eqref{originalpfaff2} is defined as follows.
\begin{definition}
\label{katz2}
Given system \eqref{originalpfaff2} whose $x,y$-\textit{exponential parts} are denoted by  ${\{\bar{a}_k\}}_{1 \leq k \leq n}$ and ${\{\bar{b}_k\}}_{1 \leq k \leq n}$ respectively. For $1 \leq k \leq n$, let $\zeta_k$ (resp. $\eta_k$) be the minimum exponent in $x$ (resp. $y$) within the terms of $\bar{a}_k$ (resp. $\bar{b}_k$).   The \textit{Katz invariant} of \eqref{originalpfaff2} is then the two-tuple of rational numbers $(\kappa_1, \kappa_2)$ where 
$$\begin{cases} \kappa_1 = - 1 - {min}_{1 \leq k \leq n} \; \zeta_k \\ \kappa_2 = - 1 - {min}_{1 \leq k \leq n} \; \eta_k .\end{cases}$$
\end{definition}

In this section, we show that the $x,y$-\textit{exponential parts} of system \eqref{originalpfaff2} are those of the two associated ODS defined below. Hence, the computation of the $x,y$-\textit{exponential parts} (consequently of $Q_1$, $Q_2$, and Katz invariant) is reduced to computations over a univariate field. 
\begin{definition}
Given system \eqref{originalpfaff2}. Let $\textbf{A}(x) := A (x,0)$ and $\textbf{B}(y) := B (0,y)$. We call the following the \textit{associated ODS} with the first and second subsystem respectively of \eqref{originalpfaff2}:
\begin{eqnarray} \label{associatedpfaff1} x \frac{d Y}{d x} &=& \textbf{A} (x) \; Y \\ \label{associatedpfaff2} y \frac{d Y}{d y} &=& \textbf{B} (y) \; Y \end{eqnarray}
\end{definition}
\begin{theorem}
\label{exponentialpfaff}
Given the completely integrable system \eqref{originalpfaff2}
$$
\begin{cases}
x \frac{\partial Y}{\partial x} = A (x,y) Y= x^{-p} (A_0(y) + A_1(y) x + \dots) Y\\ 
y \frac{\partial Y}{\partial y} = B (x,y)  Y=  y^{-q} ({B}_0(x) + {B}_1(x) y + \dots) Y.
\end{cases} $$
The $x,y$-\textit{exponential parts} of this system are those of its \textit{associated ODS}  \eqref{associatedpfaff1} and  \eqref{associatedpfaff2} respectively. \end{theorem}
\begin{proof}
We prove the theorem for $A(x,y)$. The same follows for $B(x,y)$ by interchanging the order of the subsystems in system \eqref{originalpfaff2}. Let $s_1, t$ and $T$ be as in Proposition \ref{gerardtransformation}.  Upon the change of independent variable $x=t^{s_1}$, the first subsystem of \eqref{originalpfaff2} is given by
\begin{equation}\label{firstpfaff}   \frac{\partial Y}{\partial t} =  s_1 t^{-1} A (t^{s_1},y) Y . \end{equation}
By change of basis $Y = T Z$ we arrive at the equivalent subsystem 
\begin{equation} \label{secondpfaff}  \frac{\partial Y}{\partial t} = \tilde{A}(t, y) Y \end{equation}
with the notations and properties as in Proposition \ref{gerardtransformation}.

It follows from \eqref{gaugepfaff} that
\begin{equation} \label{relation}  \frac{\partial T}{\partial t} = s_1 t^{-1} A T - T \tilde{A}. \end{equation}
On the other hand, we have the following expansions as formal power series in $y$ whose matrix coefficients lie in $\mathbb{C}((t))^{n \times n}$: 
$$ A(t,y)= \sum^{\infty}_{j=0} A_j(t) y^j, \; \tilde{A}(t,y)= \sum^{\infty}_{j=0} \tilde{A}_j(t) y^j,$$
\begin{equation} \label{expansion} \text{and} \quad T(t,y)= \sum^{\infty}_{j=0} T_j(t) y^j \end{equation} with leading terms $\textbf{A}(t^{s_1}),  \tilde{\textbf{A}},$ and  $\textbf{T}$ respectively.  We Plug \eqref{expansion} in \eqref{relation} and compare the like-power terms. In particular, we are interested in the relation between the leading terms which is clearly given by  
\begin{equation} \label{firstrelation} \frac{\partial \textbf{T}}{\partial t} = s_1 t^{-1} \textbf{A} \textbf{T} - \textbf{T} \tilde{\textbf{A}}. \end{equation} 
Due to the form of $T(t,y)$ characterized in Theorem \ref{gerardtransformation}, it is evident that $\textbf{T} \in GL_n(\mathbb{C}((t)))$. Hence, the systems given by $ \frac{\partial}{\partial t} - s_1 t^{-1} \textbf{A}(t^{s_1})$ (resp. $x\frac{\partial}{\partial x}-  \textbf{A}(x)$) and $\frac{\partial}{\partial t} -\tilde{\textbf{A}}$ are \textit{equivalent}. It follows that they have the same $x$-\textit{exponential parts} $\bar{a}_k (t)$ given in Remark \ref{exp}.
 \end{proof}

Two corollaries follow directly from Theorem \ref{exponentialpfaff}.
\begin{corollary}
Let $\kappa=(\kappa_1, \kappa_2)$ be as in Definition \ref{katz2} denoting the \textit{Katz invariant} of system \eqref{originalpfaff2}. Then $\kappa_1$ (resp. $\kappa_2$) is the \textit{Katz invariant} of the \textit{associate} ODS \eqref{associatedpfaff1} (resp. \eqref{associatedpfaff2}).
\end{corollary}
\begin{corollary}
\label{rank}
Let $\gamma=(\gamma_1, \gamma_2)$ denote the \textit{true Poincar\'e rank} of system \eqref{originalpfaff2}. Then $\gamma_1$ (resp. $\gamma_2$) is the \textit{true Poincar\'e rank} of the \textit{associate} ODS  \eqref{associatedpfaff1} (resp. \eqref{associatedpfaff2}).
\end{corollary}
\begin{proof} 
Consider system \eqref{associatedpfaff1}. By (Remark 3, \citep{key24}) we have $\gamma_1 -1 \leq \kappa_1 \leq \gamma_1.$ This fact establishes the proof since  $\gamma_1$ is an integer. The same holds for system \eqref{associatedpfaff2}. 
\end{proof}

Hence, the \textit{Katz invariant}, the \textit{true Poincar\'e rank}, and most importantly $Q_1, Q_2$ in \eqref{solutionpfaff2}, can be computed efficiently over univariate fields using the existing package ISOLDE \citep{key27}. Moreover, since now we know $\kappa$ at any stage of formal reduction, it is left to give a \textit{compatible} Moser-based rank reduction of system \eqref{originalpfaff2}. We remark however, that although the rank reduction of \citep{key5} reduces the \textit{Poicar\'e rank} to its minimal integer value, Moser-based rank reduction results in an \textit{equivalent} system for which the \textit{Poincar\'e rank} and the rank of the leading matrix coefficient are both minimal. In the univariate case, this is a necessary component of formal reduction, in particular when computing the \textit{Katz invariant}.
\section{Moser-based Rank Reduction}
\label{moser}
We consider again system \eqref{originalpfaff2} and label its subsystems as follows:
 \begin{eqnarray}
\label{pfaffian1}
x \frac{\partial Y}{\partial x} = A(x, y) Y = x^{-p} (A_0(y) + A_1(y) x + \dots) Y\\
\label{pfaffian2}
y \frac{\partial Y}{\partial y} = B(x, y) Y =  y^{-q} ({B}_0(x) + {B}_1(x) y + \dots) Y.
\end{eqnarray}
Let $T \in GL_n(K)$. To keep track of applied transformations, we use the following notation for the \textit{equivalent} system resulting upon the change of basis $Y= T Z$.
\begin{eqnarray}
\label{pfaffiannew1}
x \frac{\partial Z}{\partial x} &=& \tilde{A} Z, \quad T[A]:= \tilde{A}=T^{-1} (A T -x\frac{\partial T}{\partial x})\\
\label{pfaffiannew2}
y \frac{\partial Z}{\partial y} &=& \tilde{B} Z, \quad T[B]:= \tilde{B}= T^{-1} (B T -y\frac{\partial T}{\partial y}).
\end{eqnarray}
In particular, $T^{-1} A T$ will be referred to as the \textit{similarity term} of $T[A]$.

We study Moser-based rank reduction of subsystem \eqref{pfaffian1}. The same results follow for subsystem \eqref{pfaffian2} by interchanging \eqref{pfaffian1} its order in system \eqref{originalpfaff2}. We adapt the algorithm given for singularly-perturbed linear differential systems in \citep{key102} since it is well-suited to bivariate fields. It suffices to verify that the transformations in the proposed algorithm of system \eqref{pfaffian1} are compatible with the second. We will see that this can be guaranteed by giving an additional structure to $A_0(y)$, in particular form \eqref{gaussformpfaffian}, rather than the form proposed in (Lemma 1, \citep{key102}).  We remark that the algorithm of \citep{key102} is a generalization of Moser-based rank reduction developed by the second author of this article in \citep{key41} over a univariate field.  We recall that, in the sequel, we drop $x$ and $y$ from the notation whenever ambiguity is unlikely to arise.  Set $r= rank (A_0(y))$. 

Following \citep{key41,key19} we define the Moser rank  and Invariant of system \eqref{pfaffian1} as the respective rational numbers:
$$m(A) = \; \text{max} \;  (0, p + \frac{r}{n})$$ 
$$\quad \mu (A) = \; \text{min} \; \{ m(T[A]) | T \in GL_n(K) \}. $$
\begin{definition}
The system \eqref{pfaffian1} (the matrix $A$ respectively) is called Moser-reducible if $m(A) > \mu (A)$, otherwise it is said to be Moser-irreducible.
\end{definition}
It is easy to see from this definition that system \eqref{pfaffian1} is regular if and only if $\mu(A) \leq 1$, i.e. the \textit{true Poincar\'e rank} is zero. 
\begin{definition}
System \eqref{originalpfaff2} is said to be Moser-irrreducible whenever each of its subsystems \eqref{pfaffian1} and \eqref{pfaffian2} is. 
\end{definition}
The following theorem is the analog of (Theorem 1, \citep{key102}).
\begin{theorem}
\label{moserpfaffian}
Given System \eqref{pfaffian1} such that $A_0(y)$ is of rank $r$ and  $m(A)>1$.
A necessary and sufficient condition for $A$ to be Moser-reducible, i.e. for the existence of a $T(x, y) \in Gl_n(K)$ such that $r (\tilde{A}_0) < r$, is that the polynomial 
$$\theta(\lambda) := {x}^{r} \; det(\lambda I + \frac{A_0}{x} + A_1 )|_{x=0}$$ 
vanishes identically in $\lambda$. 
Moreover, $T(x, y)$ can always be chosen to be compatible with system \eqref{pfaffian2}. More precisely, it is a product of transformations in $GL_n(\mathcal{O}_y)$ and polynomial transformations of the form $diag (x^{\alpha_1}, \dots, x^{\alpha_n})$ where $\alpha_1, \dots, \alpha_n$ are nonnegative integers. 
\end{theorem}
\begin{remark}
By Corollary \ref{rank}, the \textit{true Poincar\'e rank} of system \eqref{originalpfaff2} can be deduced from its \textit{associated ODS}. However, the reduction criterion in Theorem \ref{moserpfaffian} guarantees that the rank of the leading coefficient matrix of the \textit{equivalent} Moser-irreducible system is minimum as well, not only its Poincar\'e rank. Moreover, this criterion furnishes the construction of the change of basis $T(x,y)$ as will be demonstrated in the sequel. 
\end{remark}
Theorem \ref{moserpfaffian} is to be proved after giving its necessary building blocks. We start by the following lemma. 
\begin{lemma}
\label{gauss1}
There exists a unimodular transformation $U(y) \in GL_n(\mathcal{O}_y)$ such that for the resulting \textit{equivalent} system \eqref{pfaffiannew1}, we have
\begin{equation} \label{gaussformpfaffian} \tilde{A}_0 (y) =  \begin{bmatrix}\tilde{A}_0^{11} & O & O\\ \tilde{A}_0^{21} & O_{r-d} & O \\ \tilde{A}_0^{31} & \tilde{A}_0^{32} & O_{n-r}  \end{bmatrix}\end{equation} 
where $$\begin{bmatrix} \tilde{A}_0^{11} \\ \tilde{A}_{0}^{21}  \end{bmatrix} \quad \text{and} \quad  \begin{bmatrix} \tilde{A}_0^{11} & O \\ \tilde{A}_0^{21} & O_{r-d}  \\ \tilde{A}_0^{31} & \tilde{A}_0^{32}  \end{bmatrix}$$ are  $r \times d$ and $n \times r$ matrices of full column ranks $d$ and $r$ respectively.  
\end{lemma}
\begin{proof}
It suffices to apply the unimodular transformation of (Lemma 1, \citep{key102}) to $A_0(y)$ and then to the first left block of the resulting similar matrix. Denoting respectively by $U_1(y) , U_2(y)$ these transformations,  we set
 
 $U(y)= diag (U_2(y), I_{n-r}).U_1(y)$. Hence, the leading coefficient matrix of $ U[A]$ has the form \eqref{gaussformpfaffian}. Clearly, $U(y)$ is \textit{compatible} with system \eqref{pfaffian2} as it is unimodular. 
\end{proof}

Hence, without loss of generality, we assume that $A_0(y)$ is in form \eqref{gaussformpfaffian} and partition $A_1(y)$ conformally with $A_0$. Let
\begin{equation} \label{glambdaformpfaffian} G_{\lambda}(A)= \begin{bmatrix} A_0^{11} & O & A_1^{13} \\ A_0^{21} & O & A_1^{23}  \\ A_0^{31} & A_0^{32} & A_1^{33} + \lambda I_{n-r}\end{bmatrix} .\end{equation}
Then we have the following Lemma given and proved as (Lemma 2, \citep{key102}).
\begin{lemma}
\label{glambda}
$Det(G_{\lambda}(A) \equiv 0$ vanishes identically in $\lambda$ if and only if $\; \theta(\lambda)$ does.   
\end{lemma}
\begin{proposition}
\label{gauss3pfaffian}
Suppose that  $m(A) >1$ and $det(G_{\lambda}(A)) \equiv 0$ where $G_{\lambda}(A)$ is given by \eqref{glambdaformpfaffian}. Then there exists a unimodular matrix $Q(y)$ in $GL_n(\mathcal{O}_y)$ with $det \; Q(y) = \pm 1$, \textit{compatible} with system \eqref{pfaffian2}, such that the matrix $G_{\lambda}(Q[A])$ has the form 
\begin{equation} \label{particularform3pfaffian} G_{\lambda}(Q[A]) = \begin{bmatrix} A_0^{11} & O & U_1^{11} & U_2^{11} \\ A_0^{21} & O & U_1^{21} & U_2^{21} \\ V_1^{11} & V_1^{12} & W_1 + \lambda I_{n- r -\rho} & W_2 \\ M_1^{11} & M_1^{12} & M_2 & W_3 + \lambda I_\rho \end{bmatrix} ,\end{equation}
where $0 \leq \rho \leq n-r,\; W_1,\; W_3$ are square matrices of orders $(n-r-\rho)$ and $\rho$ respectively , $M_1^{12}$ is a null matrix, and
\begin{eqnarray} \label{particularconditionbpfaffian} rank \; \begin{bmatrix} A_0^{11} & U_1^{11}\\ A_0^{21}  & U_1^{21} \\ M_1^{11} & M_2 \end{bmatrix} = rank \; \begin{bmatrix} A_0^{11} & U_1^{11}\\ A_0^{21} &  U_1^{21} \end{bmatrix}, \\ \label{particularconditionapfaffian} rank \; \begin{bmatrix} A_0^{11} & U_1^{11}\\ A_0^{21} &  U_1^{21}  \end{bmatrix} < r . \end{eqnarray}
\end{proposition}
\begin{proof}
Since $Q(y)$ is independent form $x$, it follows from \eqref{pfaffiannew1} that the discussion can be restricted to the similarity term of the transformations. Hence, the transformation $Q(y)$ can be constructed as in the proof of (Proposition 2, \citep{key102}). As it is unimodular, it is necessarily \textit{compatible} with system \eqref{pfaffian2}. It remains to remark however, that each row of $[M_1^{11} \quad M_1^{12} \quad M_2]$ is a linear combination of the rows of 
$$\begin{bmatrix} A_0^{11} & O & U_1^{11} \\ A_0^{21} & O & U_1^{21} \end{bmatrix}.$$ Hence, by construction, $M_1^{12}$ is a null matrix. \end{proof}

\begin{proposition}
\label{shearingpfaffian}
If $m(A) >1$ and $G_{\lambda}(A) \equiv 0$ is as in \eqref{particularform3pfaffian} with conditions \eqref{particularconditionbpfaffian} and \eqref{particularconditionapfaffian} satisfied, then system \eqref{pfaffian1} (resp. $A$) is Moser-reducible and reduction can be carried out with the so-called shearing $Y = S Z$ where $S=diag(x I_r , I_{n-r-\rho}, x I_\rho)$ if $\rho \neq 0$ and $S=diag(x I_r , I_{n-r})$ otherwise. Furthermore, this shearing is \textit{compatible} with system \eqref{pfaffian2}. 
\end{proposition}
\begin{proof}
We partition $A(x, y)$ conformally with   \eqref{particularform3pfaffian} 
$$A= \begin{bmatrix} A^{11}  & A^{12} & A^{13} & A^{14} \\ A^{21}  & A^{22} & A^{23} & A^{24}\\  A^{31}  & A^{32} & A^{33} & A^{34}\\  A^{41}  & A^{42} & A^{43} & A^{44}\end{bmatrix}$$
where $A^{11} , A^{22} , A^{33} , A^{44}$ are of dimensions $d, r-d, n-r-\rho, \rho$ respectively. 
It is easy to verify then, by \eqref{pfaffiannew1} and \eqref{pfaffiannew2}, that 
\begin{eqnarray*} \tilde{A}(x, y) &=& S^{-1} A S - x S^{-1} \frac{\partial}{\partial x} S = \begin{bmatrix} A^{11}  & A^{12} & x^{-1} A^{13} & A^{14} \\   A^{21}  & A^{22} &  x^{-1}  A^{23} & A^{24} \\   x A^{31}  & x A^{32} &   A^{33} & x A^{34} \\  A^{41}  & A^{42}  &  x^{-1} A^{43} & A^{44} \end{bmatrix}\\
& -& \; diag (I_{r}, O_{n-r-\rho},   I_\rho) \\ 
\tilde{B}(x, y) &=& S^{-1} B S  = \begin{bmatrix} B^{11}  & B^{12} & x^{-1} B^{13} & B^{14} \\   B^{21}  & B^{22} &  x^{-1}  B^{23} & B^{24} \\   x B^{31}  & x B^{32} &   B^{33} & x B^{34} \\  B^{41}  & B^{42}  &  x^{-1} B^{43} & B^{44} \end{bmatrix} .\end{eqnarray*}
Hence, the new leading coefficient matrix is 
$$ \tilde{A}_0 = \begin{bmatrix} A_0^{11} & O & U_1^{11}& O \\ A_0^{21} & O & U_1^{21} & O \\ O& O &O& O \\ M_1^{11} & O & M_2 &  O \end{bmatrix} $$
where $rank (\tilde{A}_0)  < r$ since  \eqref{particularconditionbpfaffian} and \eqref{particularconditionapfaffian} are satisfied.

It remains to prove the compatibility of $S$ with the subsystem \eqref{pfaffian2}, in particular, that the normal crossings is preserved. It suffices to prove that the submatrices of $B(x,y)$ which are multiplied by $x^{-1}$, i.e. $B^{13}, B^{23}, B^{43}$, have no term independent from $x$ so that no poles in $x$ are introduced. This can be restated as requiring $B^{13}(0,y), B^{23}(0,y),$ and $ B^{43}(0,y)$ to be null. This requirement is always satisfied due to the integrability condition \eqref{conditionpfaff2}. In fact, we can obtain from the former that
\begin{equation}
\label{integrability00}
B(0,y) A_0(y) - A_0(y) B(0,y) = y\frac{\partial A_0(y)}{\partial y}.
\end{equation}
On the other hand, since $G_{\lambda}(A)$ is as in \eqref{particularform3pfaffian} then $A_0(y)$ have the following form  \eqref{form0} and $B(0,y)$ can be partitioned conformally
\begin{equation} \label{form0} A_0(y)=\begin{bmatrix} A_0^{11} & O & O & O \\ A_0^{21} & O & O & O \\ V_1^{11} & V_1^{12} & O & O \\ M_1^{11} & O & O & O \end{bmatrix} .\end{equation}
\begin{equation} \label{form1} B(0,y)= \begin{bmatrix} B_{0y}^{11}  & B_{0y}^{12} & B_{0y}^{13} & B_{0y}^{14} \\ B_{0y}^{21}  & B_{0y}^{22} & B_{0y}^{23} & B_{0y}^{24}\\ B_{0y}^{31}  & B_{0y}^{32} & B_{0y}^{33} & B_{0y}^{34}\\  B_{0y}^{41}  & B_{0y}^{42} & B_{0y}^{43} & B_{0y}^{44} \end{bmatrix} .\end{equation}
Inserting \eqref{form0} and \eqref{form1} in \eqref{integrability00}, one can obtain the desired results by equating the entries of \eqref{integrability00}. In particular, upon investigating the entries of the L.H.S. in (Column 3), (Rows 1 and 2, Column 2), and (Row 4, Column 2), we observe the following respectively: 
\begin{itemize}
\item $\begin{bmatrix} A_0^{11} & O \\ A_0^{21} & O  \\ V_1^{11} & V_1^{12}  \\ M_1^{11} & O  \end{bmatrix} \begin{bmatrix} B_{0y}^{13} \\ B_{0y}^{23}  \end{bmatrix} =  O_{n, n-\rho - r} $. The first matrix is of full rank $r$ by construction thus $ \begin{bmatrix} B_{0y}^{13} \\ B_{0y}^{23}  \end{bmatrix}$ is null. 
\item $\begin{bmatrix} A_0^{11} \\ A_0^{21}  \end{bmatrix} B_{0y}^{12} = O_{r, r-d}$. The first matrix is of full rank $d$ by construction thus $B_{0y}^{12}$ is null.
\item Finally, $B_{0y}^{43} \; V_1^{12} - M_1^{11} \; B_{0y}^{12}  =O_{\rho, (r-d)}$. But $B_{0y}^{12}$ is null and $V_1^{12}$ is of full column rank $r-d$ by construction and so $B_{0y}^{43}$ is null as well.  
\end{itemize} This completes the proof. \end{proof} 

We give hereby the proof of Theorem \ref{moserpfaffian}.

\begin{proof} (Theorem \ref{moserpfaffian}) For the necessary condition, we proceed as in the proof of Theorem 1 in \citep{key102}. As for the sufficiency, we set $r= rank (A_0(y))$. Without loss of generality, we can assume that $A_0(y)$ has the form \eqref{gaussformpfaffian}. Let $G_{\lambda}(A)$ be given as in \eqref{glambdaformpfaffian}. Then, by Lemma \ref{glambda}, $det(G_{\lambda}(A))$ vanishes identically in $\lambda$ if and only if $\; \theta(\lambda)$ does. Then the matrix $S[Q[A]]$ where $S, Q$ are as in Propositions \ref{gauss3pfaffian} and \ref{shearingpfaffian} respectively, has the desired property. 
\end{proof}
\begin{corollary}
\label{corollarymoserpfaffian}
Given system \eqref{originalpfaff2} with \textit{Poincar\'e rank} $(p,q)$.  A necessary and sufficient for it to be Moser-reducible is that one of the polynomials 
$$\begin{cases} \theta_A (\lambda) :={x}^{rank(A_0)} \; det(\lambda I + \frac{A_0}{x} + A_1 )|_{x=0} \\  \theta_B(\lambda) := {y}^{rank(B_0)} \; det(\lambda I + \frac{B_0}{y} + B_1 )|_{y=0} \end{cases}$$ 
vanishes identically in $\lambda$. An equivalent Moser-irreducible system can be attained via a change of basis compatible with both subsystems \eqref{pfaffian1} and \eqref{pfaffian2}. 
\end{corollary}

We illustrate the process by this simple example.
\begin{example}
\label{simple}
\begin{equation}
\label{exm}
\begin{cases}
x \frac{\partial Y}{\partial x} = x^{-3} \begin{bmatrix} x^3+ x^2+y & y^2 \\ -1 & x^3 + x^2 -y \end{bmatrix} Y \\
y \frac{\partial Y}{\partial y} = 
y^{-2} \begin{bmatrix} y^2 -2 y -6 & y^3 \\ -2 y & -3 y^2 -2 y -6 \end{bmatrix}Y 
\end{cases} \end{equation}

The \textit{associated} ODS are then respectively:
$$\begin{cases}x \frac{\partial Y}{\partial x} = x^{-3} \begin{bmatrix} x^3+ x^2 & 0 \\ -1 & x^3 + x^2 \end{bmatrix} Y,\\
y \frac{\partial Y}{\partial y} = y^{-2} \begin{bmatrix} y^2 -2 y -6 & y^3 \\ -2 y & -3 y^2 -2 y -6 \end{bmatrix}Y. \end{cases}$$

Via ISOLDE, we compute $Q_1=\frac{-1}{x} I_2 $ and $Q_2=\frac{3}{y^2} + \frac{2}{y} I_2. $ Thus we have $s_1 = s_2 = 1$ and \eqref{solutionpfaff2} is given by
$$ \Phi(x,y) \; x^{{\Lambda}_1} \; y^{{\Lambda}_2} \; e^{\frac{-1}{x}} \; e^{\frac{3}{y^2} + \frac{2}{y}}.$$

Upon the eigenvalue shifting $Y=  e^{\frac{-1}{x}} e^{\frac{3}{y^2} + \frac{2}{y}}  Z$ we get from \eqref{exm}
$$
\begin{cases}
x \frac{\partial Z}{\partial x} = x^{-3} \begin{bmatrix} x^3+y & y^2 \\ -1 & x^3  -y \end{bmatrix} Z  \\
y \frac{\partial Z}{\partial y} =  y^{-1} \begin{bmatrix} y  & y^2 \\ -2  & -3 y  \end{bmatrix} Z
\end{cases} $$
We arrive at the system of Example \eqref{exmnaive}. By Algorithm \ref{algorithmpfaffian}, we compute $T_1=\begin{bmatrix} y x^3 & -y \\ 0 & 1 \end{bmatrix}$. Hence, by $Z = T_1 U$, we have
$$
\begin{cases}
x  \frac{\partial U}{\partial x} = \begin{bmatrix} -2 & 0 \\ -y & 1 \end{bmatrix} U  \\
y \frac{\partial U}{\partial y} =  \begin{bmatrix} -2 & 0\\ -2 x^3 & -1 \end{bmatrix} U
\end{cases} 
$$
By a simple calculation, we find $T_2=\begin{bmatrix} 1 & 0 \\ \frac{y}{3} + 2 x^3 &-1 \end{bmatrix}$. 

A fundamental matrix of solutions is then given by  $$ T_1 T_2 \; x^{\Lambda_1} \; y^{\Lambda_2}e^{\frac{-1}{x}} e^{\frac{3}{y^2} + \frac{2}{y}}$$
where $\Lambda_1=\begin{bmatrix} -2 & 0 \\ 0 & 1 \end{bmatrix}$ and $\Lambda_2=\begin{bmatrix} -2 & 0 \\ 0 & -1 \end{bmatrix}$.
\end{example}
\begin{algorithm}
\caption{Moser-based Rank Reduction of System \eqref{originalpfaff2}}
\label{algorithmpfaffian}
\textbf{Input:} $A(x,y), B(x,y)$ of \eqref{originalpfaff2}  \\
\textbf{Output:} $T(x, y)$ a change of basis and a Moser-irreducible equivalent system \{$T[A], T[B]$\}. In particular, the Poincar\'e rank of this system is its true Poincar\'e rank.\\

\begin{algorithmic}
\State $T \gets I_{n}$; $p  \gets$ Poincar\'e rank of $A$; 
\State $U(y)  \gets $ Lemma \ref{gauss1};
\State $A \gets  U^{-1} A U$; $T  \gets  T U$;
\While {Det($G_{\lambda}(A)=0$) and $p >0$} \do \\
\State $Q(y), \rho  \gets$ Proposition \ref{gauss3pfaffian};
\State $S(x) \gets $ Proposition \ref{shearingpfaffian};
\State $P  \gets  Q S$; $T  \gets  T P$;
\State $A \gets  P^{-1} A P - x S^{-1} \frac{\partial S}{\partial x}$;
\State $p  \gets $ Poincar\'e rank of $A$;
\State  $U(y)  \gets $ Lemma \ref{gauss1};
\State $A \gets  U^{-1} A U$; $T  \gets  T U$;
\EndWhile .\\
$B \gets T^{-1} B T - yT^{-1} \frac{\partial T}{\partial y}$;\\
$T \gets I_{n}$; $q  \gets$ Poincar\'e rank of $B$; 
\State $U(x)  \gets $ Lemma \ref{gauss1};
\State $B \gets  U^{-1} B U$; $T  \gets  T U$;
\While {Det($G_{\lambda}(B)=0$) and $q >0$} \do \\
\State $Q(x), \rho  \gets$ Proposition \ref{gauss3pfaffian};
\State $S(y) \gets $ Proposition \ref{shearingpfaffian};
\State $P  \gets  Q S$; $T  \gets  T P$;
\State $B \gets  P^{-1} B P - y S^{-1} \frac{\partial S}{\partial y}$;
\State $q  \gets $ Poincar\'e rank of $B$;
\State  $U(x)  \gets $ Lemma \ref{gauss1};
\State $B \gets  U^{-1} B U$; $T  \gets  T U$;
\EndWhile .\\
$A \gets T^{-1} A T - xT^{-1} \frac{\partial T}{\partial x}$;\\
\Return{(T, A, B)}.
\end{algorithmic}
\end{algorithm}
\section{Conclusion and Further Investigations}
\label{conclusion}
We gave an explicit method to compute the  $x,y$-\textit{exponential parts}  of a \textit{completely integrable pfaffian system with normal crossings in two variables}. This gives the main information about the asymptotic behavior of its solutions. Moreover, this new approach limited the computations to a finite number of constant matrices instead of matrix-valued functions and constituted an eminent portion of the formal reduction. To complement our work, we gave a Moser-based rank reduction algorithm. Both results, allowed us to generalize the formal reduction of the univariate case as developed in \citep{key24} to the bivariate case. One field to investigate for this bivariate system would be an algorithm to construct a basis for the space of regular solutions (see, e.g., \citep{key25,key40}, for $m=1$).

Another research direction over bivariate fields is the generalization of techniques developed here to \textit{completely integrable pfaffian systems} with no restriction to the locus of singularities (i.e. general crossings rather than normal crossings). Such systems are discussed e.g. in \citep{key32}.

 And there remains of course the ultimate task of formal reduction of the multivariate system with no restriction to the number of variables. Theorem \ref{gerardexistence} was first given and proved in \citep{key3} for bivariate systems. In the theory developed there, one operator $\Delta_i$ was considered and the fact that $\mathcal{O}_{y}$ and $\mathcal{O}_{x}$ are principal ideal domains was used in many places to prove that certain modules introduced are free modules. This did not allow an immediate generalization to the case of more than two variables. The same obstacle arises in \textit{rank reduction} whether in \citep{key5} or in adapting the Moser-based rank reduction algorithm of \citep{key102,key41}. This limits our proposed formal reduction in Section \ref{moser} to $m=2$. However a generalization of Proposition \ref{gerardtransformation} is Theorem 2.3 in \citep{key53}. This furnishes the generalization of Theorem \ref{exponentialpfaff} to a general multivariate system. Furthermore, it motivates the investigation of their formal reduction since the leading coefficient matrix of a change of basis which takes the system to a weak-triangular form, would be characterized.

\bibliography{mybib}
\end{document}